\documentclass[11pt]{amsart}
\usepackage{amsmath}
\usepackage{amsfonts}
\usepackage{amsthm}
\usepackage{amssymb}
\usepackage{verbatim}
\usepackage{enumerate}
\usepackage{graphicx}
\usepackage{longtable}
\usepackage[all,cmtip]{xy}
\usepackage{upref}
\usepackage{hyperref}
\usepackage{color}
\usepackage[hang,small,bf]{caption}

\numberwithin{equation}{section}

\newtheorem{theorem}{Theorem}[section]
\newtheorem{proposition}[theorem]{Proposition}
\newtheorem{lemma}[theorem]{Lemma}
\newtheorem{corollary}[theorem]{Corollary}

\theoremstyle{definition}
\newtheorem{definition}[theorem]{Definition}
\newtheorem{remark}[theorem]{Remark}

\renewcommand{\theta}{\vartheta}

\newcommand{\T}{{\mathbb T}}
\newcommand{\Z}{{\mathbb Z}}
\newcommand{\R}{{\mathbb R}}
\newcommand{\N}{{\mathbb N}}

\newcommand{\A}{{\mathbb A}}

\renewcommand{\phi}{\varphi}

\begin{document}

\title[Periodic motions of a particle in an oscillating field on the two-torus]{On the periodic motions of a charged particle\\ in an oscillating magnetic field on the two-torus}

\author{Luca Asselle}
\address{Ruhr-Universit\"at Bochum, Fakult\"at f\"ur Mathematik, NA 4/35, Universit\"atsstra\ss e 150, D-44780 Bochum, Germany}
\email{\href{mailto:luca.asselle@ruhr-uni-bochum.de}{luca.asselle@ruhr-uni-bochum.de}}
\author{Gabriele Benedetti}
\address{Universit\"at Leipzig, Mathematisches Institut, Augustusplatz 10, D-04109 Leipzig, Germany}
\email{\href{mailto:gabriele.benedetti@math.uni-leipzig.de}{gabriele.benedetti@math.uni-leipzig.de}}


\date{\today}

\begin{abstract}
Let $(\T^2,g)$ be a Riemannian two-torus and let $\sigma$ be an oscillating $2$-form on $\T^2$. We show that for almost every small positive number $k$ the magnetic flow of the pair $(g,\sigma)$ has infinitely many periodic orbits with energy $k$. This result complements the analogous statement for closed surfaces of genus at least $2$ \cite{AB15a} and at the same time extends the main theorem of \cite{AMMP14} to the non-exact oscillating case. 
\end{abstract}

\maketitle

\section{Introduction}

Let $(M,g)$ be a closed connected orientable Riemannian surface and let $\sigma$ be a 2-form on $M$. We call $\sigma$ the \emph{magnetic form}. We denote with $\omega_g$ the standard symplectic form on $TM$ obtained by pulling back the canonical symplectic form on $T^*M$ via the Riemannian metric and with
$\omega_{g,\sigma}:= \omega_g + \pi^*\sigma$
the \textit{twisted symplectic form} determined by the pair $(g,\sigma)$, where $\pi:TM\to M$ is the projection. The Hamiltonian flow on the symplectic manifold $(TM,\omega_{g,\sigma})$ associated with the kinetic energy
\[E(q,v) = \frac{1}{2}\, |v|_q^2\]
is called the \textit{magnetic flow of the pair} $(g,\sigma)$ and the trajectories of such a flow are called \textit{magnetic geodesics}.
Indeed, this flow models the motion of a charged particle under the effect of a magnetic field represented by $\sigma$. More precisely, let $\nabla$ be the \textit{Levi-Civita connection} of $g$ and $Y_{g,\sigma}:TM\rightarrow TM$ be the \textit{Lorentz force} defined by the formula
\begin{equation*}
g_q(u,Y_{g,\sigma}(q,v))= \sigma_q(u,v),\quad\forall\, q\in M,\ \forall\, u,v\in T_qM.
\end{equation*}
A curve $\gamma:\R\rightarrow M$ solves the equation
\begin{equation}\label{eq:lor}
\nabla_{\dot \gamma}\dot\gamma= Y_{g,\sigma}(\gamma,\dot\gamma)
\end{equation} 
if and only if $(\gamma,\dot\gamma):\R\rightarrow TM$ is a trajectory of the Hamiltonian flow of $E$.

Following \cite{Arn61}, Equation \eqref{eq:lor} can be equivalently formulated in terms of the \textit{geodesic curvature} $\kappa_g$ of the curve $\gamma$. Indeed, let $\mu_g$ be the area form associated with $g$ and let $f:M\rightarrow \R$ be the unique function such that $\sigma=f\mu_g$. We call $f$ the \textit{density} of $\sigma$ with respect to $\mu_g$. For all $k>0$, a curve $\gamma:\R\rightarrow M$ with constant speed $\sqrt{2k}$ satisfies \eqref{eq:lor} if and only if
\begin{equation}
\kappa_g(\gamma)= -\,\frac{f(\gamma)}{\sqrt{2k}}.
\end{equation} 
For example, this shows that, when $M=S^2=\{x^2+y^2+z^2=1\}\subset \R^3$, $g$ is the round metric and $\sigma=\mu_g$, every magnetic geodesic is supported on the intersection of $S^2$ with some affine plane in $\R^3$ not passing through the origin.

A central question in the study of the dynamics of magnetic flows is the existence of \emph{closed} magnetic geodesics. In \cite{AMMP14} it is shown that if $\sigma=d\theta$ is exact, then for almost every $k\in(0,c_u(L_\theta))$ the energy level $E^{-1}(k)$ carries infinitely many geometrically distinct closed magnetic geodesics. Here $c_u(L_\theta)$ denotes the \textit{Ma\~n\'e critical value of the universal cover} (see \cite{Con06} or \cite{Abb13} for the precise definition) of the Lagrangian
\begin{equation}
L_\theta(q,v)= \frac{1}{2}|v|^2_q+\theta_q(v).
\label{lagrangianalocale}
\end{equation}

One of the research directions undertaken by the authors of this paper is to extend such result to the case of an \textit{oscillating} $\sigma$.
\begin{definition}
We say that $\sigma$ is $\textit{oscillating}$ if there exist $q_-,q_+\in M$ such that $\sigma_{q_-}<0$ and $\sigma_{q_+}>0$.
\end{definition}
We notice that oscillating forms are a natural generalization of the exact ones, since we can think of exact forms as ``balanced'' oscillating forms, their integral over $M$ being zero. We already showed in \cite{AB15a} that the result proved in \cite{AMMP14} for exact forms extends to oscillating forms when $M$ has genus at least $2$ and $c_u(L_\theta)$ is replaced by some $\tau_+^*(g,\sigma)\in(0,c_u(L_\theta)]$ (observe that $c_u(L_\theta)$ is still well-defined since the lift of $\sigma$ to the universal cover is exact). Implementing ideas contained in \cite{AB14}, we are now able to treat the case in which $M=\T^2$ is the two-torus. After we submitted our manuscript, the case of the two-sphere has also been solved by the authors in collaboration with Abbondandolo, Mazzucchelli and Taimanov \cite{AABMT16}.

The aim of the present paper is therefore to prove the following

\begin{theorem}\label{theorem:main}
Let $\sigma$ be an oscillating 2-form on $(\T^2,g)$. There exists a positive real number $\tau_+(g,\sigma)>0$ such that for almost every $k\in (0,\tau_+(g,\sigma))$ the energy level $E^{-1}(k)$ carries infinitely many geometrically distinct closed magnetic geodesics.
\end{theorem}

\begin{remark}
A generic 2-form $\sigma$ on $M$ is either oscillating or symplectic. The latter case has also been object of intensive research in relation with the existence of periodic orbits. Indeed, the following facts are known when $\sigma$ is symplectic. If $M\neq S^2$ there exist infinitely many closed magnetic geodesics on every low energy level \cite{FH03,GGM15}. If $M=S^2$ there are either two or infinitely many closed magnetic geodesics for every low energy level \cite{Ben14b}. Under some non-resonance conditions the second alternative holds for every low energy level \cite{Ben15b}. However, there are also examples of magnetic systems with a ``low'' energy level having exactly two closed magnetic geodesics \cite{Ben15}.
\end{remark}

\begin{remark}
In \cite[Theorem 6.1]{AM16} the existence of local minimizers has been extended to the case of magnetic Tonelli systems $(L,\sigma)$ for energies belonging to the interval $(e_0(L),e(L,\sigma))$, see \cite{AM16} for the precise definitions. This yields a generalization of Theorem \ref{theorem:main} to that setting, as explicitly stated in \cite[Theorem 6.2]{AM16}. Very little is known on periodic orbits for systems that are not Tonelli. A result asserting the existence of periodic orbits for almost every compact regular energy level of a general Hamiltonian system on $S^2$ was recently obtained in \cite{BZ15}.
\end{remark}

The proof of Theorem \ref{theorem:main} relies on a variational characterization of closed magnetic geodesics with energy $k$. They correspond to the zeros of a suitable closed 1-form $\eta_k$, called the \textit{action 1-form}, defined on the Hilbert manifold $\mathcal M$ of $H^1$-loops with arbitrary period. The main difference with respect to the case ``$\sigma$ exact'' in \cite{AMMP14}, or with the case ``$M$ has genus at least 2'' in \cite{AB15a}, is that, if $\sigma$ is a non-exact form on $\T^2$, then $\eta_k$ cannot be described as the differential of a functional on the subset of non-contractible loops. Work of Ta\u \i manov shows that we cannot circumvent this problem by simply restricting the study to the subset of contractible loops. Indeed, he recently constructed some concrete examples of oscillating magnetic forms on $\T^2$ for which the orbits given by Theorem \ref{theorem:main} cannot be contractible \cite{Tai15}.

The lack of a globally defined functional gives rise to two main difficulties. First, we cannot use the level sets of the functional to distinguish between different magnetic geodesics (this was an essential point for the arguments of \cite{AMMP14}). As we explain below, we solve this problem exploiting the fact that the fundamental group of $\T^2$ is torsion-free. Second, we need to find a class of sequences $(\gamma_h)_{h\in\N}\subset\mathcal M$, generalizing the classical Palais-Smale sequences, and for which a compactness theorem holds. This second problem is solved by making use of recent work of the authors. Indeed, in \cite[Theorem 2.1]{AB14} we showed that every vanishing sequence $(\gamma_h)$ (i.e. a sequence such that $|\eta_k(\gamma_h)|\rightarrow 0$) whose periods are bounded and bounded away from zero admits a converging subsequence. Since limiting points of vanishing sequences are zeros of $\eta_k$, the aforementioned compactness property provides the right tool to prove the existence of closed magnetic geodesics with energy $k$. 

The starting point of the proof of Theorem \ref{theorem:main} is the existence of a zero $\alpha_k$ of $\eta_k$ of a particular kind, namely a ``local minimizer'' for $\eta_k$ (see Section 
\ref{localminimizersonsurfaces} for a precise definition). The existence of such a zero of $\eta_k$ for every $k\in (0,\tau_+(g,\sigma))$ follows from \cite{Tai92b,Tai92a,Tai93} or from \cite[Appendix C]{CMP04}. Without loss of generality we might suppose that $\alpha_k$ is 
non-contractible, as otherwise the proof contained in \cite{AB15a} would go through without any change. Observe that, if $\alpha_k$ is non-contractible, then also its iterates $\alpha_k^n$ are 
non-contractible and are contained in distinct connected components, say $\mathcal N^n$, of the space $\mathcal M$. We use the $\alpha_k^n$'s
to define suitable minimax classes in each $\mathcal N^n$ and corresponding monotonically increasing minimax functions $c_n(k)$. The monotonicity 
will be an essential ingredient to show that the minimax functions actually yield zeros $\gamma_n(k)$ of $\eta_k$ for almost every $k\in (0,\tau_+(g,\sigma))$. This is the content of the \textit{Struwe monotonicity argument} \cite{Str90}\footnote{See \cite{AMMP14,Con06,Mer10} for other applications of this argument to the existence of closed orbits and \cite{Ass15a} for an application to the existence of orbits satisfying the conormal boundary conditions.}, which we generalize to our setting. The fact that the sets $\mathcal N^n$ are all distinct combined with the fact that isolated zeros for $\eta_k$ cease to be of mountain-pass type if iterated sufficiently many times (cf. \cite[Theorem 2.6]{AMMP14} or Proposition \ref{iterationofmountainpasses}) will show that the $\gamma_n(k)$'s can not be iterates of finitely many zeros of $\eta_k$, thus concluding the proof.

\noindent We end this introduction with a summary of the contents of the present work:

\begin{itemize}
\item In Section \ref{theaction1form} we introduce the action form $\eta_k$ and recall its global properties.
\item In Section \ref{localzero} we analyze the behavior of $\eta_k$ locally around a zero.
\item In Section \ref{localminimizersonsurfaces} we recall the existence, for every sufficiently low energy, of closed magnetic geodesics which are local minimizers of the action.
\item In Section \ref{theminimaxclasses} we define the minimax classes and the minimax functions.
\item In Section \ref{astruwetypemonotonicityargument} we prove the mountain pass lemma.
\item In Section \ref{elimination} we conclude the proof of the main theorem.
\end{itemize}

\subsection*{Acknowledgements}
Luca Asselle is partially supported by the DFG grant AB 360/2-1 ``Periodic orbits of conservative systems below the Ma\~n\'e critical energy value''. Gabriele Benedetti is partially supported by the DFG grant SFB 878. We are indebted to Marco Mazzucchelli for the notion of essential family. We would like to thank the anonymous referee for her careful reading of the draft and for her suggestions which helped us improving the paper.


\section{The action 1-form}
\label{theaction1form}

We start by defining the 1-form $\eta_k$ and recalling some general facts about it. For the proofs of the statements in this section we refer to \cite[Section 2]{AB14}. Let $(M,g)$ be a closed connected Riemannian manifold and let $\sigma$ be a (closed) 2-form on $M$. We will write $|\cdot|$ for the norm on $TM$ induced by $g$. Let us denote by $\mathcal M:=H^1(\T,M)\times (0,+\infty)$ the space of $H^1$-loops in $M$ with arbitrary period. We write $\mathcal M_0$ for the component of contractible loops. The space $\mathcal M$ has the structure of a Hilbert manifold endowed with a metric given by $g_{\mathcal M}=g_{H^1}+dT^2$, 
where $g_{H^1}$ is the usual Riemannian metric on $H^1(\T,M)$ induced by the metric $g$. With slight abuse of notation we will also denote the norm on $\mathcal M$ induced by $g_{\mathcal M}$ with $|\cdot |$.

Throughout this paper we will adopt the identification $\gamma=(x,T)$ where $\gamma:\R\rightarrow M$ is such that $\gamma(t)=x(t/T)$.
We have a $\T$- and an $\N$-action on $\mathcal M$, where by $\N$ we denote the set of positive integers. The former action changes the base point of the loop:
\begin{equation*}
\psi\cdot \gamma:=\big(x(\psi+\,\cdot\,),T\big),\quad\quad \forall\, \psi\in\T,\ \forall\,\gamma=(x,T)\in\mathcal M.
\end{equation*}
The latter action iterates the loop:
\begin{equation*}
\gamma^n:= (x^n,nT),\quad\quad \forall\, n\in\N,\ \forall\, \gamma=(x,T)\in\mathcal M,
\end{equation*}
where $x^n(s):= x(ns)$, $\forall\, s\in\T$.

Let us now define for every $k\in(0,+\infty)$ the \textit{action 1-form} $\eta_k\in \Omega^1(\mathcal M)$ by
\[\eta_k(x,T):= d\A_k(x,T) + \int_0^1 \sigma_{x(s)}\big(\cdot,x'(s)\big)\, ds,\]
where $\A_k: \mathcal M\rightarrow \R$ is given by
\begin{equation*}
\A_k(x,T):= T\cdot \int_0^1\Big(\frac{1}{2T^2}|x'(s)|^2 +k\Big)\,ds = \frac{e(x)}{T}+ kT
\end{equation*}
and $e(x)$ is the kinetic energy of $x$
\begin{equation*}
e(x):= \frac12 \int_0^1 |x'(s)|^2 \, ds.
\end{equation*}

The action 1-form is smooth and closed. It is, furthermore, $\T$-invariant. In particular, the set of zeros of $\eta_k$ is the disjoint union of sets of the type $\T\cdot\gamma$. We call every such a set a \textit{vanishing circle}. Moreover, an element $\gamma\in\mathcal M$ is a zero of $\eta_k$ if and only if $\gamma$ is a closed magnetic geodesic with energy $k$. 
In view of this, we will find closed magnetic geodesics with energy $k$ by constructing zeros of $\eta_k$. We will achieve this goal using an approximation procedure.

\begin{definition}
We call $(\gamma_h)=(x_h,T_h)\subset\mathcal M$ a \emph{vanishing sequence} for $\eta_k$, if 
\[|\eta_k(\gamma_h)| \rightarrow 0 .\]
\end{definition}

By continuity $\eta_k$ vanishes on the set of limit points of vanishing sequences. So we are led to ask: which vanishing sequences do have a non-empty  limit set? Clearly, if $T_h\rightarrow 0$ or $T_h\rightarrow\infty$, then the limit set is empty. The following theorem shows that the converse is also true.

\begin{theorem}\label{theorem:ps}
Let $(\gamma_h)=(x_h,T_h)$ be a vanishing sequence for $\eta_k$ in a given connected component of $\mathcal M$ with $T_h\leq T^*<\infty$ for every $h\in \N$. Then the following statements hold:
\begin{enumerate}
\item If $T_h$ tends to zero, then $\ e(x_h)\rightarrow 0$.
\item If $0<T_*\leq T_h$, $\forall \, h\in \N$, then $(\gamma_h)$ has a converging subsequence. 
\end{enumerate}
\end{theorem}

We are going to find vanishing sequences by considering certain minimax classes of paths in $\mathcal M$. For our argument we need a vector field $X_k$ generalizing the negative gradient of the free-period action functional. It is defined by
\begin{equation}\label{eq:xk}
X_k:= \frac{-\,\sharp\eta_k}{\sqrt{1+|\eta_k|^2}},
\end{equation}
where $\sharp$ is the duality between $T\mathcal M$ and $T^*\mathcal M$ induced by the Riemannian metric $g_{\mathcal M}$. Let $\Phi^{k}$ be the positive semi-flow of $X_k$. It is known that the flow lines of $\Phi^k$ that blow up in finite time go closer and closer to the subset of constant loops. Hence, the restriction of the semi-flow $\Phi^k$ to any connected component $\mathcal N\neq\mathcal M_0$ of $\mathcal M$ is \textit{complete}, namely all its trajectories are defined for all positive times. Moreover, by the definition of $X_k$ we have the following consequence of Theorem \ref{theorem:ps}.
\begin{corollary}\label{cor:ps}
Let $\mathcal N\neq\mathcal M_0$ be a connected component of $\mathcal M$, $T^*$ a positive real number, and $\mathcal Y'\subset \mathcal N$ a neighborhood of the zeros of $\eta_k$ that are contained in the set $\mathcal N\cap\{T\leq T^*\}$. There exists $\varepsilon=\varepsilon(T^*,\mathcal Y')>0$ such that
\begin{equation}
\big ( \mathcal N\cap \{T\leq T^*\}\big )\setminus \mathcal Y'\ \subset\ \{|\eta_k(X_k)|\geq\varepsilon\}.
\end{equation}
\end{corollary}


\section{Local properties of the action 1-form on surfaces}
\label{localzero}

In this section we analyze some local properties of the action 1-form $\eta_k$ under the assumption that $M$ is a closed connected orientable surface.  

If $\gamma$ is a zero of $\eta_k$ we now construct a neighborhood of $\T\cdot\gamma^\N$ where the action form admits a well-behaved primitive. Since $\gamma$ is a smooth curve, $\gamma(\R)\subseteq M$ is a set of zero Lebesgue measure. In particular, there exists $q\in M\setminus \gamma(\R)$. We set $V^\gamma:=M\setminus \{q\}$ and observe that $\sigma|_{V^\gamma}$ is exact.
Let $\theta^\gamma\in\Omega^1(V^\gamma)$ be a primitive of $\sigma$ on $V^\gamma$. We denote by $\mathcal V^\gamma$ the open subset of $\mathcal M$ made by the loops with image entirely contained in $V^\gamma$ and observe that $\mathcal V^\gamma$ is an open neighborhood of the set $\T \cdot\gamma^\N$. Furthermore, $\eta_k$ is exact on $\mathcal V^\gamma$ \textit{for every} $k\in(0,+\infty)$ with primitive ${\mathbb S}^\gamma_k:\mathcal V^\gamma\rightarrow\R$ given by the formula
\begin{equation*}
{\mathbb S}^\gamma_k(x,T):=T\cdot \int_0^1\Big[L_{\theta^\gamma}\Big (x(s),\frac{x'(s)}{T}\Big )+k\Big]\,ds,\quad  L_{\theta^\gamma}(q,v):= \frac12 |v|_q^2 + \theta^\gamma_q(v).
\end{equation*}
Namely, ${\mathbb S}^\gamma_k$ is the free-period action functional associated with the Lagrangian $L_{\theta^\gamma}$. Since ${\mathbb S}^\gamma_k$ belongs to the class of functionals considered in \cite{AMMP14}, we can translate Theorem 2.6 contained therein to our setting. Notice indeed that the base manifold does not have to be compact for that result to hold. 

\begin{proposition}
Let $k>0$ and let $\gamma\in\mathcal M$ be such that for every $\nu\in \N$, $\T\cdot \gamma^\nu$ is an isolated vanishing circle. Let ${\mathbb S}_k^\gamma:\mathcal V^\gamma\rightarrow \R$ be the local primitive of $\eta_k$ defined above. There exists $\nu(\gamma)\in\N$ such that for all $\nu> \nu(\gamma)$ the following holds: There exists a fundamental system of open neighborhoods $\mathcal W_\nu\subseteq \mathcal V^\gamma$ of $\T \cdot \gamma^\nu$ such that, if $\beta_0, \beta_1 \in \{{\mathbb S}_k^\gamma<{\mathbb S}_k^\gamma(\gamma^\nu)\}\cap\mathcal W_\nu$ are contained in the same connected component of $\mathcal W_\nu$ then $\beta_0$ and $\beta_1$ are contained in the same connected component of $\{{\mathbb S}_k^\gamma<{\mathbb S}_k^\gamma(\gamma^\nu)\}$.
\label{iterationofmountainpasses}
\end{proposition}
\begin{remark}
Given $\beta_0$ and $\beta_1$ as above, we stress that the path connecting them inside $\{{\mathbb S}_k^\gamma<{\mathbb S}_k^\gamma(\gamma^\nu)\}$ might not be contained in $\mathcal W_\nu$.
\end{remark}
We now move to consider zeros of $\eta_k$ of a particular type.

\begin{definition}
We say that $\alpha\in \mathcal M$ is a \emph{local minimizer of the action} (with energy $k$) if there exists an open neighborhood $\mathcal U^\alpha\subseteq \mathcal V^\alpha$ of $\T\cdot \alpha$ such that 
\begin{equation}\label{eq:locmin}
{\mathbb S}_k^\alpha(\gamma) \geq {\mathbb S}_k^\alpha(\alpha) , \quad \forall\,\gamma \in \mathcal U^\alpha .
\end{equation}
The local minimizer $\alpha$ is called \emph{strict}, if inequality \eqref{eq:locmin} is strict $\forall\, \gamma\in \mathcal U^\alpha\setminus \T\cdot\alpha$.
\label{localminimizers/strict}
\end{definition}
The next proposition, which follows from \cite[Lemma 3.1]{AMP13}, states that the property of being a local minimizer is preserved under iterations.
\begin{proposition}
If $\alpha$ is a (strict) local minimizer of the action, then for every $n\geq 1$ the $n$-th iterate $\alpha^n$ is also a (strict) local minimizer of the action.
\label{persistenceoflocalminimizers}
\end{proposition}

The last proposition of this section gives more information on the neighborhood $\mathcal U^\alpha$, when the minimizer $\alpha$ is strict, see \cite[Lemma 4.3]{AMP13} for the proof.

\begin{proposition}\label{prp:minnei}
If $\alpha$ is a strict local minimizer of the action with energy $k$, there exists an open neighborhood $\mathcal U^\alpha$ of $\T\cdot \alpha$ such that
\begin{equation}
\inf_{\partial \mathcal U^\alpha}{\mathbb S}_k^\alpha > {\mathbb S}_k^\alpha(\alpha).
\label{inequalitystrictlocalminimizer}
\end{equation}
\label{strictlocalminimizer}
\end{proposition}


\section{Local minimizers for the action 1-form on surfaces}
 
\label{localminimizersonsurfaces}

We now recall the existence of local minimizers when $(M,g)$ is an orientable closed connected Riemannian surface and $\sigma$ is a $2$-form on it. Up to changing the orientation of $M$, we can also assume that the integral of $\sigma$ over $M$ is non-negative.
Let $\mathcal F_+$ be the space of positively oriented  embedded surfaces in $M$ (in \cite{Tai92b,Tai92a,Tai93} Ta\u \i manov considers the so-called \textit{films}). We remark that the elements in $\mathcal F_+$ can have boundary or more than one connected component and that the empty surface $\varnothing$ also belongs to $\mathcal F_+$. If $k\in(0,+\infty)$ we consider the family of Ta\u\i manov functionals
\begin{equation}
\mathcal T_k:\mathcal F_+\rightarrow \R,\quad\quad\mathcal T_k(\Pi) := \sqrt{2k}\cdot l(\partial \Pi)+ \int_\Pi \sigma,
\label{Taimanovfunctional}
\end{equation}
where $l(\partial \Pi)$ denotes the length of the boundary of $\Pi$. We readily find that
\begin{equation}
\mathcal T_k (\varnothing) = 0 , \quad \mathcal T_k(M) = \int_{M} \sigma \geq 0.
\label{positivi}
\end{equation}
The family $k\mapsto\mathcal T_k$ is increasing and each $\mathcal T_k$ is bounded from below since
\[\mathcal T_k(\Pi)\geq -\Vert \sigma\Vert_\infty\cdot \operatorname{area}_g(M) .\]

We define now the value
\[\tau_+(M,g,\sigma):= \sup \big \{ k\ \big  |\ \inf \mathcal T_k< 0\big\} .\]
The functionals $\mathcal T_k$ can be lifted to any finite cover $p':M'\rightarrow M$, giving rise to the set of values $\tau_+(M',g,\sigma)$. We define the \textit{Ta\u\i manov critical value} as
\begin{equation}
\tau_+(g,\sigma):= \sup\Big\{ \tau_+(M',g,\sigma)\ \Big |\ p':M'\rightarrow M \mbox{ finite cover}\Big\} .
\end{equation}
In \cite{CMP04} it was shown that, when $\sigma=d\theta$ is exact, the Ta\u\i manov critical value coincides with $c_0(L_\theta)$, the Ma\~n\'e critical value of the abelian cover of the Lagrangian $L_\theta$ as in \eqref{lagrangianalocale}. 
To our knowledge there is no such a precise characterization for a general $\sigma$. However, a finite upper bound for $\tau_+(g,\sigma)$ in terms of suitable Ma\~n\'e critical values can still be found and $\tau_+(g,\sigma)$ turns out to be strictly positive if and only if $\sigma$ is oscillating (see \cite[Lemma 6.4]{AB15a} for a proof of the latter property). As far as the first assertion is concerned, let us consider any 2-form $\sigma'$ on $M$ such that $\sigma'\geq 0$ everywhere and the difference $\sigma-\sigma'$ is exact. If $\theta$ is any primitive of $\sigma-\sigma'$, we readily observe that
\begin{equation*}
\mathcal T_k(\Pi)=\sqrt{2k}\cdot l(\partial \Pi)+ \int_\Pi d\theta+ \int_\Pi \sigma' \geq \sqrt{2k}\cdot l(\partial \Pi)+ \int_\Pi d\theta.
\end{equation*}
From this inequality and from the characterization of $\tau_+(g,\sigma)$ in the exact case, it follows that $\tau_+(g,\sigma)\leq c_0(L_\theta)$. Thus, we conclude that $\tau_+(g,\sigma)$ is finite and, more precisely,
\begin{equation*}
\tau_+(g,\sigma)\leq\inf_{\sigma'}\inf_{d\theta=\sigma-\sigma'}c_0(L_\theta),
\end{equation*}
where $\sigma'$ is a 2-form on $M$ such that $\sigma'\geq 0$ and $\sigma-\sigma'$ is exact.

We can now state the main theorem about the existence of local minimizers of the action. This result is an easy corollary of Ta\u\i manov's theorem \cite{Tai92a,Tai93,CMP04} about the existence of global minimizers for $\mathcal T_k$ and Section 3 in \cite{AMP13}. For a detailed discussion when $M$ has genus at least $2$, we refer to \cite[Section 6]{AB15a}.

\begin{theorem}
Let $g$ be a Riemannian metric on a closed connected orientable surface $M$ and $\sigma \in \Omega^2(M)$ an oscillating form. For every $k\in(0,\tau_+(g,\sigma))$ there exists a closed magnetic geodesic $\alpha_k$ on $M$ with energy $k$ which is a local minimizer of the action. 
\label{Taimanov92}
\end{theorem}


\section{The minimax classes}
\label{theminimaxclasses}

From now on let $\T^2$ be the two-torus endowed with a Riemannian metric $g$ and an oscillating 2-form $\sigma$. For the rest of the paper we fix $k^*\in (0,\tau_+(g,\sigma))$. By Theorem \ref{Taimanov92}, there exists a local minimizer with energy $k^*$, which we denote by $\alpha_{k^*}$. We notice that, if $\alpha_{k^*}$ is not strict, then there exists a sequence of local minimizers of the action approaching $\alpha_{k^*}$, which are all closed magnetic geodesics with energy $k^*$. Thus, from now on we suppose that the local minimizer $\alpha_{k^*}$ is strict. To prove Theorem \ref{theorem:main} we need only show the following

\begin{proposition}\label{lem:main}
There exists an open interval $I=I(k^*)\subset (0,\tau_+(g,\sigma))$ containing $k^*$ such that for almost every $k\in I$, there exist infinitely many geometrically distinct closed magnetic geodesics with energy $k$.
\end{proposition}
The proof of the proposition will occupy us for the remaining of the paper.

If $\alpha_{k^*}$ is contractible then we restrict the study to $\mathcal M_0$. Here, the action 1-form $\eta_k$ becomes exact since $\pi_2(\T^2)=0$ (cf. \cite{Mer10}) with primitive given by 
\begin{equation*}
\A_k(x,T) + \int_{C(x)} \sigma,
\end{equation*} 
where $C(x)$ is any capping disc for $x$.
In this case, Proposition \ref{lem:main} follows by reproducing the same argument as in \cite{AB15a}, having in mind the compactness stated in Theorem \ref{theorem:ps}.

Hence, hereafter we assume that $\alpha_{k^*}$ is non-contractible and for every $n\in\N$ we denote by $\mathcal N^n$ the connected component of $\mathcal M$ containing $\alpha_{k^*}^n$. Recall that by Proposition \ref{strictlocalminimizer} we can find a \textit{bounded} open neighborhood $\mathcal U=\mathcal U^{\alpha_{k^*}}\subseteq \mathcal V^{\alpha_{k^*}}$ of $\T\cdot \alpha_{k^*}$ such that 
\begin{equation*}
\inf_{\partial \mathcal U}{\mathbb S}_{k^*}^{\alpha_{k^*}}>{\mathbb S}_{k^*}^{\alpha_{k^*}}(\alpha_{k^*}).
\end{equation*}
Here $\mathcal V^{\alpha_{k^*}}$ is the open neighborhood of $\T\cdot\alpha_{k^*}^\N$ and ${\mathbb S}_{k^*}^{\alpha_{k^*}}$ is the local primitive of $\eta_{k^*}$ defined at the beginning of Section \ref{localzero}. For all $k\in(0,\tau_+(g,\sigma))$ we define
\begin{equation}
M_k:= \overline{\Big \{ \text{local minimizers of} \ {\mathbb S}_k^{\alpha_{k^*}} \ \text{in} \ \mathcal U\Big \}}.
\label{Mk}
\end{equation}

We now find an interval of energies around $k^*$ where the sets $M_k$ have a good behavior. We refer to \cite[Lemma 3.1 and 3.2]{AMMP14} for the proof.

\begin{lemma}\label{rococo}
There exists an open interval $I=I(k^*)\subset (0,\tau_+(g,\sigma))$ containing $k^*$ and with the following properties:
\begin{enumerate}[(i)]
\item For every $k\in I$ the set $M_k$ is non-empty and compact.
\item For every pair $k_-<k_+ \in I$ and for every $\beta\in M_{k_+}$ there exists a continuous path $w:[0,1]\rightarrow \mathcal U$ such that $w(0)\in M_{k_-}$, $w(1)=\beta$ and 
\begin{equation}\label{eq:rococo}
{\mathbb S}_{k_-}^{\alpha_{k^*}}\circ w \leq {\mathbb S}_{k_-}^{\alpha_{k^*}}(\beta).
\end{equation}
\end{enumerate}	
\end{lemma}

Before defining the minimax classes we need some information about the topology of $\mathcal M$, which we now recall.
For every $(a,b)\in\Z^2$ denote by $\mathcal M(a,b)$ the connected component of the loops $\gamma$ such that $\widetilde{\gamma}(T)=(a,b)+\widetilde{\gamma}(0)$, where $T$ is the period of $\gamma$ and $\widetilde \gamma$ is any lift of $\gamma$ to $\R^2$. Let $\mathcal M(a,b)^n$ be the connected component of $\mathcal M$ containing the $n$-th iterates of the elements in $\mathcal M(a,b)$. Since clearly $\mathcal M(a,b)^n= \mathcal M(na,nb),$
the sets $\mathcal M(a,b)^{n}$ are pairwise disjoint for ${(a,b)}\neq (0,0)$. We set $\gamma_{(a,b)}\in\mathcal M(a,b)$ for every $(a,b)\in\Z^2$ to be the projection to $\T^2$ of the path in $\R^2$ given by $t\mapsto (ta/T, tb/T )$. It is a classical fact that
\begin{equation*}
\phi_{(a,b)}:\T^2\rightarrow\mathcal M(a,b)\,,\quad\quad \phi_{(a,b)}(q):= q+ \gamma_{(a,b)}
\end{equation*}
is a homotopy equivalence whose homotopy inverse is the evaluation at zero.

Let now $\bar\sigma\in\Omega^2(\T^2)$ be given by $\bar\sigma=dq^1\wedge dq^2$, where $(q^1,q^2)$ are the Cartesian coordinates in $\R^2$. We can associate to it the \textit{transgression 1-form} $\tau\in\Omega^1(\mathcal M)$ given by 
\begin{equation*}
\tau_{\gamma}:=\int_0^T\bar\sigma_{\gamma(t)}\big(\,\cdot\,,\dot\gamma(t)\big)\,dt,\quad \gamma\in\mathcal M.
\end{equation*}
Such a form is closed and, hence, it yields a cohomology class $[\tau]\in H^1(\mathcal M,\R)$.
\begin{lemma}\label{lem:top}
Let $(a,b)\in\Z^2$, then:
\begin{enumerate}[(i)]
\item We have
\begin{equation}\label{eq:pull}
[\phi_{(a,b)}^*\tau]=-\,[\imath_{(a,b)}\bar\sigma]=\big[\,bdq^1\,-\,adq^2\,\big]\in H^1(\T^2,\R).
\end{equation}
Thus, the restricted class $[\tau]|_{\mathcal M(a,b)}$ is trivial if and only if ${(a,b)}=(0,0)$.
\item For every closed magnetic geodesic $\gamma$ the restriction 
$[\tau]|_{\mathcal V^\gamma}$ is trivial, where $\mathcal V^\gamma$ is the neighborhood of $\T\cdot \gamma^\N$ 
defined in Section \ref{localzero}.
\end{enumerate}
\end{lemma}

\begin{proof}
Let $(q,v)\in T\T^2\simeq \T^2\times\R^2$ and compute 
\begin{equation*}
\tau_{\phi_{(a,b)}(q)}\big(d_q\phi_{(a,b)}(v)\big)=\int_0^T\bar\sigma\left (v, \frac 1T{(a,b)}\right )\,dt =\bar\sigma\big(v,(a,b)\big),
\end{equation*}
which yields at once \eqref{eq:pull}. As far as the second statement is concerned, if $\bar\theta^\gamma\in\Omega^1(V^\gamma)$ is a primitive of $\bar\sigma|_{V^\gamma}$, then a primitive of $\tau$ on $\mathcal V^\gamma$ is given by
\begin{equation*}
\beta \mapsto \int_0^{T}\beta^*\bar\theta^\gamma.
\end{equation*}
\end{proof}

\begin{corollary}\label{cor:disc}
For any $k\in(0,+\infty)$ and $(a,b)\in\Z^2$ we have
\begin{equation}
[\phi^*_{(a,b)}\eta_k]= \left(\int_{\T^2}\sigma\right)\cdot \big[bdq^1-adq^2\big] \in H^1(\T^2,\R).
\end{equation}
In particular, the image $[\eta_k]\big(H_1(\mathcal M(a,b),\Z)\big)$ is a discrete subgroup of $\R$.
\end{corollary}

We can now proceed to the definition of the sequence of minimax classes we are interested in. For every $k\in I$ and $n\in\N$  we set
\begin{equation}\label{miniclasses}
\mathcal P_n(k):= \Big\{u:[0,1]\rightarrow \mathcal N^n\ \Big|\ u(0)=u(1)\in M_k^n,\ \int_0^1u^*\tau\neq0 \Big\},
\end{equation} 
where $M^n_k$ is made of the $n$-th iterates of the elements in $M_k$. Since $\mathcal N^n\neq\mathcal M_0$ the set $\mathcal P_n(k)$ is non-empty by Lemma \ref{lem:top}.(i). Intuitively we can think of $\mathcal P_n(k)$ as the set of loops in $\mathcal N^n$ based at some point of $M_k^n$ whose homotopy class in $\pi_1(\mathcal N^n)$ is not contained in the subgroup generated by the $\T$-action, which changes the base point of a loop. 
By Lemma \ref{lem:top}.(ii) we have that all the elements in $\mathcal P_n(k)$ have to leave a suitable neighborhood of $M^n_k$. This last property will be crucial in the proof of Lemma \ref{Struwesurface}.

\begin{remark}\label{r:tail}
If $n\in\N$ and $k_-<k_+$, then the minimax classes $\mathcal P_n(k_-)$ and $\mathcal P_n(k_+)$ are disjoint. However, there is a natural way to relate them. Namely, if $u_+\in\mathcal P_n(k_+)$ and $\beta^n=u_+(0)=u_+(1)$ for some $\beta\in M_{k_+}$, then Lemma \ref{rococo}.(ii) yields $w:[0,1]\to \mathcal U$ with $w(0)\in M_{k_-}$, $w(1)=\beta$ and such that \eqref{eq:rococo} holds. We denote by $w^n,\bar w^n:[0,1]\to\mathcal U^n$ the paths defined as $w^n(s):=(w(s))^n$ and $\bar w^n(s):=(w(1-s))^n$, for all $s\in[0,1]$. Thus, we obtain an element $u_-\in \mathcal P_n(k_-)$ as $u_-:=w^n\ast u_+\ast \bar w^n$, where $\ast$ denotes the concatenation of paths. 
\end{remark}

For every $u\in \mathcal P_n(k)$ we now set 
\[ {\mathbb S}_k(u,s):=  {\mathbb S}_k^{\alpha_{k^*}}\big(u(0)\big) + \int_0^s u^*\eta_k, \quad \forall\, s\in [0,1]\]
and define $c_n:I\rightarrow \R$ by
\begin{equation}
c_n(k) := \inf_{u\in \mathcal P_n(k)} \max_{s\in [0,1]} {\mathbb S}_k(u,s).
\label{minimaxfunctionetaknotexact}
\end{equation}
The relation between the functions ${\mathbb S}$ for two different values of the energy is described in the next lemma, cf. \cite[Lemma 4.2]{AB14} for the proof.
\begin{lemma}\label{lem:mono}
Let $u=(x,T):[0,1]\rightarrow\mathcal M$ be a continuous path such that $u(0)\in\mathcal V^{\alpha_{k^*}}$. If $k_-$ and $k_+$ are positive real numbers, we have
\begin{equation}\label{eq:intvar}
{\mathbb S}_{k_+}(u,s)= {\mathbb S}_{k_-}(u,s)+ (k_+-k_-)T(s),\quad\forall\, s\in[0,1].
\end{equation}
\end{lemma}

An important consequence of Remark \ref{r:tail} and Lemma \ref{lem:mono} is the monotonicity of the minimax functions.

\begin{lemma}
If $n\in \N$ and $k_-,k_+\in I$ with $k_-<k_+$, then $c_n(k_-)\leq c_n(k_+)$. 
\label{monotonicityetakexact}
\end{lemma}
\begin{proof}
Consider $u_+\in \mathcal P_n(k_+)$ and let $u_-=w^n\ast u_+\ast \bar w^n\in\mathcal P_n(k_-)$ be the path defined in Remark \ref{r:tail}. In particular, we have $u_-|_{[1/3,2/3]}\ =\ u_+\big (3\, (\cdot -1/3)\big )$ and, if $\beta^n:=w^n(0)=w^n(1)=u_+(0)=u_+(1)$, then
\begin{equation}
u_-\big ([0,1/3]\big ) = w^n\big ([0,1]\big ) \subseteq \big \{ {\mathbb S}_{k_-}^{\alpha_{k^*}}\leq {\mathbb S}_{k_-}^{\alpha_{k^*}}(\beta^n)\},
\label{grugrugru}
\end{equation}
since $w$ satisfies \eqref{rococo} and ${\mathbb S}_{k_-}^{\alpha_{k^*}}$ is $\N$-equivariant. Analogously,
\begin{equation}
u_-\big ([2/3,1]\big ) =\bar w^n\big ([0,1]\big )\subseteq \big \{ {\mathbb S}_{k_-}^{\alpha_{k^*}}\leq {\mathbb S}_{k_-}^{\alpha_{k^*}}(\beta^n)\}.
\label{grugrugru'}
\end{equation}
Thus, if $s\in[0,1/3]$, \eqref{grugrugru} implies that
\begin{equation*}
{\mathbb S}_{k_-}(u_-,s)= {\mathbb S}_{k_-}^{\alpha_{k^*}}(u_-(s))\leq {\mathbb S}_{k_-}^{\alpha_{k^*}}(\beta^n) \leq {\mathbb S}_{k_+}^{\alpha_{k^*}}(u_+(0))={\mathbb S}_{k_+}(u_+,0).
\end{equation*}
If $s\in [1/3,2/3]$, applying \eqref{eq:intvar} we get
\begin{align*}
{\mathbb S}_{k_-}(u_-,s) = {\mathbb S}_{k_-}^{\alpha_{k^*}}(u_+(0)) + \int_{0}^{3(s-1/3)} \hspace{-20pt}u_+^*\eta_{k_-} &={\mathbb S}_{k_-}\big (u_+,3(s-1/3)\big )\\
&\leq {\mathbb S}_{k_+}\big (u_+,3(s-1/3)\big ).
\end{align*}
Finally, if $s\in [2/3,1]$, \eqref{grugrugru'} implies that
\begin{align*}
{\mathbb S}_{k_-}(u_-,s) &= {\mathbb S}_{k_-}^{\alpha_{k^*}}(u_+(0)) + \int_0^1 u_+^*\eta_{k_-} + \int_{2/3}^{s} u_-^*\eta_{k_-}\\
&= {\mathbb S}_{k_-}(u_+,1) + {\mathbb S}_{k_-}^{\alpha_{k^*}}(u_-(s)) - {\mathbb S}_{k_-}^{\alpha_{k^*}}(\beta^n) \\
& \leq {\mathbb S}_{k_+}(u_+,1).
\end{align*}
Summarizing, we have that 
\[c_n(k_-) \leq\max_{s\in [0,1]} {\mathbb S}_{k_-}(u_-,s)\leq \max_{s\in [0,1]} {\mathbb S}_{k_+}(u_+,s)\]
and hence taking the infimum over $u_+\in \mathcal P_n(k_+)$ we see that $c_n(k_-)\leq c_n(k_+)$.
\end{proof}


\section{Essential families and the mountain-pass lemma}
\label{astruwetypemonotonicityargument}

In the previous section we have defined the minimax classes and the corresponding minimax functions. In order to relate these objects to the zeros of $\eta_k$, we introduce now the notion of essential family. 
\begin{definition}\label{d:ess}
Let $\mathcal E=\cup_{j\in J}\T\cdot\gamma_j$ be a union of isolated vanishing circles for $\eta_k$ on $\mathcal N^n$, where $J$ is a set of indices. A separating neighborhood $\mathcal W=\cup_{j\in J}\mathcal W_j$ for $\mathcal E$ is a union of disjoint connected open sets $\mathcal W_j$ such that $\T\cdot\gamma_j\subset \mathcal W_j\subset \mathcal V^{\gamma_j}$, for all $j\in J$. We say that $\mathcal E$ is an essential family for $\mathcal P_n(k)$ if for all separating neighborhoods $\mathcal W$ for $\mathcal E$, there exists $u\in\mathcal P_n(k)$ and $\lambda>0$ such that
\begin{align*}\label{eq:ess}
\mathbf{(A)}&\quad \{0,1\}\subset\big\{ {\mathbb S}_k(u,\,\cdot\,)\leq c_n(k)-\lambda\big\},\\
\mathbf{(B)}&\quad \forall\,s\in\big\{{\mathbb S}_k(u,\,\cdot\,)>c_n(k)-\lambda\big\},\ \exists\, j\in J\ \text{such that}\\ 
&\quad \mathbf{(B1)}\ \ u(s)\in\mathcal W_j,\qquad \mathbf{(B2)}\ \ c_n(k)-{\mathbb S}_k(u,s)={\mathbb S}_k^{\gamma_j}(\gamma_j)-{\mathbb S}^{\gamma_j}_k(u(s))\,.
\end{align*}
\end{definition}

\begin{remark}\label{r:empty}
By the definition of $c_n(k)$ every essential family is non-empty.
\end{remark}
\begin{remark}\label{r:pri}
Condition (B2) might seem a bit mysterious at first sight. However, it just says that $\gamma_j$ is a critical point at level $c_n(k)$ for the local primitive of $\eta_k$ defined as
\begin{equation*}
{\mathbb S}^{\gamma_j,u,s}_k : \mathcal W_j\to\R,\qquad {\mathbb S}^{\gamma_j,u,s}_k(\beta):= {\mathbb S} ^{\gamma_j}_k(\beta)+\Big({\mathbb S}_k(u,s)-{\mathbb S}^{\gamma_j}_k\big(u(s)\big)\Big).
\end{equation*}
Namely, condition \textbf{(B2)} is equivalent to
\begin{equation*}
\mathbf{(B2')}\quad {\mathbb S}^{\gamma_j,u,s}_k(\gamma_j)=c_n(k).
\end{equation*}
The quantity ${\mathbb S}^{\gamma_j,u,s}_k (\beta)$ can also be defined as the sum of ${\mathbb S}^{\alpha_{k^∗}}_k (u(0))$ and the integral of $\eta_k$ over the concatenation of $u|_{[0,s]}$ with any path connecting $u(s)$ to $\beta$ within $\mathcal W_j$.
\end{remark}

We can now establish the mountain pass lemma, whose proof relies on the celebrated Struwe's monotonicity argument \cite{Str90}, and that produces a finite essential family for $\mathcal P_n(k)$, whenever $n$ and $k$ satisfy some special hypotheses described in the statement below and discussed further in Remark \ref{oss:hyp}.

Thus, let $\alpha_{k^*}$ be a non-contractible strict local minimizer belonging to a connected component $\mathcal N\neq \mathcal M_0$ of $\mathcal M$. Let $\mathcal U\subset \mathcal V^{\alpha_{k^*}}$ be a bounded neighborhood of $\T\cdot\alpha_{k^*}$ satisfying \eqref{inequalitystrictlocalminimizer} and let $I\subset (0,\tau_+(g,\sigma))$ be the open interval containing $k^*$ given by Lemma \ref{rococo}. For every $k\in I$, let $M_k$ be the subset of $\mathcal U$ as in \eqref{Mk}, $\mathcal P_n(k)$ be the minimax class as in \eqref{miniclasses} and $c_n(k)$ be the minimax value as in \eqref{minimaxfunctionetaknotexact}. 
\begin{lemma}\label{Struwesurface}
Let $n\in \N$ and $k\in I$ be such that $c_n:I\to \R$ is differentiable at $k$. Let $T_*>0$ be a Lipschitz constant for $c_n$ at $k$ such that
\begin{equation}\label{eq:lip}
T_* > n\cdot \sup_{\mathcal U}\ T-2
\end{equation}
and such that the vanishing circles for $\eta_k$ contained in 
$\mathcal N^n\cap\{T\leq T_*+3\}$ form a discrete set. If we call $\mathcal E_n(k)$ the union of such vanishing circles, then $\mathcal E_n(k)$ is essential for $\mathcal P_n(k)$.
\end{lemma}

\begin{remark}\label{oss:hyp}
The hypotheses that we put on $n$ and $k$ are very natural in view of Proposition \ref{lem:main}. First, by the Lebesgue differentiation theorem the function $c_n:I\rightarrow\R$ is differentiable almost everywhere since it is monotone by Lemma \ref{monotonicityetakexact}. Second, if $\mathcal N^n\cap\{T\leq T_*+3\}$ contains infinitely many vanishing circles, then the existence of infinitely many geometrically distinct closed magnetic geodesics with energy $k$ trivially follows.
\end{remark}

\begin{proof}
We divide the argument into four steps.
\medskip

\textbf{Step 1.} Choose a strictly decreasing sequence $k_h\downarrow k$ and set $\lambda_h:=k_h-k$. Without loss of generality we may suppose that for all $h\in \N$, there holds
\begin{equation}
 c_n(k_h)-c_n(k) \leq T_{*}\lambda_h .
\label{lipschitzcontinuityk}
\end{equation}
For every $h\in \N$ choose $u_h=(x_h,T_h)\in \mathcal P_n(k_h)$ such that 
\[\max_{s\in [0,1]} {\mathbb S}_{k_h}(u_h,s) < c_n(k_h) + \lambda_h.\]
\noindent Let $s\in[0,1]$ be such that ${\mathbb S}_{k}(u_h,s) >c_n(k)-\lambda_h$. It follows 
\[T_h(s) = \frac{ {\mathbb S}_{k_h}(u_h,s) - {\mathbb S}_k(u_h,s)}{\lambda_h} < \frac{(c_n(k_h)+\lambda_h)-(c_n(k)-\lambda_h)}{\lambda_h} \leq T_* + 2\]
and at the same time, using \eqref{lipschitzcontinuityk},
\begin{equation*}
{\mathbb S}_k(u_h,s) \leq {\mathbb S}_{k_h}(u_h,s) < c_n (k)+ (T_*+1) \lambda_h .
\end{equation*}
Thus, $\forall\,h\in \N$ and $\forall\,s \in[0,1]$ one between conditions $(a)$ and $(b)$ below holds:
\begin{align}
(a)&\quad {\mathbb S}_{k}(u_h,s) \leq c_n(k)- \lambda_h ,\label{primaalternativak}\\
(b)&\quad {\mathbb S}_k(u_h,s) \in \Big(c_n(k)-\lambda_h, c_n(k)+(T_*+1)\lambda_h\Big), \ \ T_h(s)<T_*+2 .\label{secondaalternativak}
\end{align}

By Remark  \ref{r:tail} applied to $k_-=k$ and $k_+=k_h$, we can form the concatenated paths $w_h^n\ast u_h\ast \bar w^n_h$, which belong to $\mathcal P_n(k)$. Thanks to \eqref{eq:lip} and the properties of $w_h$ prescribed by Lemma \ref{rococo}.(ii), $w_h^n\ast u_h\ast \bar w^n_h$ satisfy the above dichotomy, as well. By a slight abuse of notation we will also denote such paths by $u_h$, so that, from now on, $u_h$ will be an element of $\mathcal P_n(k)$. Moreover, by \eqref{eq:lip} the set $M_k$ is the union of a finite number of vanishing circles, which are then strict local minimizers of the action. Up to taking a subsequence and using the $\T$-action on loops, we may also suppose that all the paths $u_h$ start from and end at the same strict local minimizer $\xi\in M_k^n$.
\medskip

\textbf{Step 2.} Let $\Phi^k$ denote the semi-flow of the bounded vector field $X_k$ conformally equivalent to $-\sharp \eta_k$ defined in \eqref{eq:xk}. As we saw in Section \ref{theaction1form}, the restriction of such a semi-flow to $\mathcal N^n$ is complete since $\mathcal N^n\neq\mathcal M_0$. Hence, we can consider for every $r\in[0,1]$ the element $u_h^r\in\mathcal P_n(k)$ given by
\begin{equation*}
u_h^r(s):=\Phi^k_r(u_h(s)), \quad \forall\, s\in[0,1].
\end{equation*}
Since $\eta_k$ is a closed form, for all $r,s\in[0,1]$, Stokes' Theorem yields the relation
\begin{equation}\label{eq:dec}
\mathbb S_k(u_h^r,s)=\mathbb S_k(u_h,s) + \int_0^r \eta_k\left(\frac{\partial }{\partial r'}u_h^{r'}(s)\right)d r'\,.
\end{equation}
As $\tfrac{\partial }{\partial r'}u_h^{r'}(s)=X_k$, the second term on the right is non-positive. Hence, the map $r'\mapsto {\mathbb S}_{k}(u^{r'}_h,s)$ is decreasing. Combining this fact with \eqref{primaalternativak} and \eqref{secondaalternativak}, we see that, for all $s\in[0,1]$, one between conditions $(a')$ and $(b')$ below holds:
\begin{align*}
(a')&\quad {\mathbb S}_{k}(u_h^1,s)\leq  c_n(k)-\lambda_h,\\
(b')&\quad {\mathbb S}_{k}(u_h^r,s) \in \Big(c_n(k) - \lambda_h, c_n(k) + (T_*+1)\lambda_h\Big),\quad \forall\, r\in[0,1].
\end{align*}
Suppose that $s\in[0,1]$ satisfies the second alternative. We get that
\begin{equation}\label{eq:below2}
{\mathbb S}_{k}(u_h^r,s)-  {\mathbb S}_{k}(u_h,s)> (c_n(k) - \lambda_h) - \big(c_n(k)+(T_*+1)\lambda_h\big)= - (T_*+2) \lambda_h.
\end{equation}
This inequality enables us to bound the variation of the period along a flow line. Indeed, in general, it holds
\begin{equation*}
|T_h^r(s)-T_h(s)|^2\leq r\cdot\left(-\int_0^r \eta_k\left(\frac{\partial }{\partial r'}u_h^{r'}(s)\right)d r'\right),\quad\forall\, r\in[0,1].
\end{equation*}
Combining this estimate with \eqref{eq:dec} and \eqref{eq:below2}, we obtain
\[T^r_h(s) \leq \big |T^r_h(s)-T_h(s)\big | + T_h(s) \leq \sqrt{r(T_*+2)\,\lambda_h}+(T_*+2) < T_*+3,\]
where the last inequality is true for $h$ sufficiently large.
\medskip

\textbf{Step 3.} We claim that for every neighborhood $\mathcal Y$ of the family $\mathcal E_n(k)$ and for all $h$ sufficiently large, the following implication holds:
\begin{equation}
\forall\, s\in[0,1],\quad {\mathbb S}_k(u^1_h,s)> c_n(k)-\lambda_h\quad\Longrightarrow\quad u^1_h(s) \in \mathcal Y.
\end{equation}
Consider a neighborhood $\mathcal Y'\subset \mathcal Y$ of $\mathcal E_n(k)$ such that $\Phi^k_r(\mathcal Y')\subset \mathcal Y$ for all $r\in[0,1]$. We apply Corollary \ref{cor:ps} to $T_*+3$ and we find $\varepsilon>0$ such that
\begin{equation}\label{eq:below}
\{T\leq T_*+3\}\setminus \mathcal Y' \subset\big\{|\eta_k(X_k)|\geq\varepsilon\big\}.
\end{equation}
Suppose that, for some $s\in[0,1]$, 
\[{\mathbb S}_k(u^1_h,s)> c_n(k)- \lambda_h,\quad u^1_h(s)\notin \mathcal Y.
\] 
Then, $u^r_h(s)\in \{T\leq T_*+3\}\setminus\mathcal Y'$ for all $r\in[0,1]$. Combining Estimate \eqref{eq:below} with Equation \eqref{eq:dec} we find a contradiction to \eqref{eq:below2} for $h$ large enough:
\begin{equation*}
{\mathbb S}_{k}(u_h^1,s)- {\mathbb S}_{k}(u_h,s) \leq - \varepsilon. 
\end{equation*} 

\textbf{Step 4.} Let $\mathcal W$ be a separating neighborhood for $\mathcal E_n(k)$. We claim that there exists $h_{\mathcal W}$ such that if $h\geq h_{\mathcal W}$, then $u^1_h$, $\lambda_h$ and $\mathcal W$ satisfy conditions \textbf{(A)} and \textbf{(B)} in Definition \ref{d:ess}.

Let $\xi$ be the element in $M_k^n$ given by Step 1. Since $\xi$ is a strict local minimizer, by Proposition \ref{prp:minnei} it has a neighborhood $\mathcal U^\xi\subset \mathcal V^{\alpha_{k^*}}$ satisfying \eqref{inequalitystrictlocalminimizer}. By Lemma \ref{lem:top}, $u^1_h([0,1])$ is not contained in $\mathcal U^\xi$ for any $h\in\N$. In particular, there exist a smallest $s^-_h$ and a largest $s^+_h$ in $[0,1]$ such that $u_h^1(s^\pm_h)\in\partial \mathcal U^\xi$ and hence there exists $\rho>0$ such that 
\[
{\mathbb S}_k^{\alpha_{k^*}} (\xi) +\rho < {\mathbb S}_k^{\alpha_{k^*}}(u_h^1(s_h^\pm)).
\]
Using the fact that $\max {\mathbb S}_k(u_h^1,\cdot)\leq c_n(k)+(T_*+1)\lambda_h$, we now compute
\begin{align*}
{\mathbb S}_k(u^1_h,0) &= {\mathbb S}_k(u_h^1,s_h^-) - \int_0^{s^-_h} (u_h^1)^*\eta_k\\
				     &\leq c_n(k)+(T_*+1)\lambda_h + {\mathbb S}_k^{\alpha_{k^*}}(\xi)-{\mathbb S}_k^{\alpha_{k^*}} (u_h^1(s_h^-))\\
				     &< c_n(k)+(T_*+1)\lambda_h-\rho.
\end{align*}
Analogously, we get ${\mathbb S}_k(u^1_h,1)<c_n(k)+(T_*+1)\lambda_h-\rho$ and for $h$ large 
\begin{equation*}
\max\big\{{\mathbb S}_k(u^1_h,0) , {\mathbb S}_k(u^1_h,1)\big\} \leq c_n(k)-\lambda_h.
\end{equation*}

Now that \textbf{(A)} is established, we move to consider \textbf{(B1)}. Step 3 implies
\[
u_h\Big(\big\{{\mathbb S}_k(u_h^1,\,\cdot\,)>c_n(k)\big\}\Big)\subset\mathcal W,\quad \mbox{for $h$ large}.
\]
To show \textbf{(B2)} we need to prove the following statement: Given $\T\cdot\gamma\in\mathcal E_n(k)$ and the connected component $\mathcal W_\gamma$ of $\mathcal W$ such that $\gamma\in\mathcal W_\gamma$, there holds
\[
{\mathbb S}_k^{\gamma,u_h^1,s_h}(\gamma)=c_n(k),\quad\forall\, s_h\in \big\{{\mathbb S}_k(u^1_h,\,\cdot\,)>c_n(k)-\lambda_h\big\}\cap (u_h^1)^{-1}(\mathcal W_\gamma),
\]
for $h$ large and with ${\mathbb S}^{\gamma,u_h^1,s_h}$ defined as in Remark \ref{r:pri}.
Thus, let $s_h\in[0,1]$ be such that ${\mathbb S}_k(u_h^1,s_h)>c_n(k)-\lambda_h$ and $u_h^1(s_h)\in\mathcal W_\gamma$. We set
\[a_h\,:=\ {\mathbb S}_k^{\gamma,u_h^1,s_h}(\gamma)\,.\]
First, we claim that there is $\delta>0$ such that if $a_h\neq a_{h'}$, for $h,h'\in\N$, then
\[
|a_h-a_{h'}|\geq\delta.
\]
Indeed, thanks to Remark \ref{r:pri}, there exists $v_{h,h'}:\T\to \mathcal N^n$ such that
\[
a_h-a_{h'}=\int_\T v_{h,h'}^*\eta_k.
\]
The path $v_{h,h'}$ is the concatenation of four paths: the restriction $u_h^1|_{[0,s_h]}$, a path connecting $u^1_h(s_h)$ to $\gamma_j$ inside $\mathcal W_j$, a path connecting $\gamma_j$ to $u^1_{h'}(s_{h'})$ inside $\mathcal W_j$ and the restriction $u_{h'}^1|_{[0,s_{h'}]}$ traversed in the opposite direction. 
The claim then follows from Corollary \ref{cor:disc}. Thus, if we prove that $a_h\rightarrow c_n(k)$, we find $a_h=c_n(k)$ for $h$ large, as we wanted. By definition we have
\[
a_h - c_n(k)=\Big({\mathbb S}_k(u^1_h,s_h)-c_n(k)\Big)+\Big({\mathbb S}_k^\gamma(\gamma)-{\mathbb S}_k^\gamma\big(u_h^1(s_h)\big)\Big).
\]
The first term on the right tends to zero since $\big|{\mathbb S}_k(u^1_h,s_h)-c_n(k)\big| \leq (T_*+1)\lambda_h$, whereas the second term tends to zero since ${\mathbb S}^\gamma_k$ is a continuous function and $u_h^1(s_h)\to \gamma$. This very last statement follows by applying Step 3 to an \emph{arbitrary} separating neighborhood of $\mathcal E_n(k)$. We conclude that the sequence $(a_h)$ tends to $c_n(k)$, thus proving Step 4 and the entire lemma.  
\end{proof}

\section{The elimination lemma and the proof of the Main Theorem}\label{elimination}
In this section we prove Theorem \ref{theorem:main} by establishing Proposition \ref{lem:main} through the following elimination lemma.

\begin{lemma}\label{lem:el}
Let $\mathcal E$ be any essential family for $\mathcal P_n(k)$ and let $\gamma\in \mathcal M$ be such that $\T\cdot\gamma^\nu$ is an isolated vanishing circle of $\eta_k$, for all $\nu\in \N$. Let $\nu(\gamma)$ be the positive integer defined in Proposition \ref{iterationofmountainpasses}. If $\nu>\nu(\gamma)$, then $\mathcal E':=\mathcal E\setminus\T\cdot\gamma^\nu$ is also an essential family for $\mathcal P_n(k)$. 
\end{lemma}
\begin{proof}
Fix $\nu>\nu(\gamma)$ and let $\mathcal W_\nu\subset\mathcal V^\gamma$ be an arbitrarily small connected neighborhood of $\T\cdot\gamma^\nu$ such that the condition given by Proposition \ref{iterationofmountainpasses} holds.

Let $\mathcal W'$ be a separating neighborhood for $\mathcal E'=\mathcal E\setminus\T\cdot\gamma^\nu$. Up to shrinking the neighborhoods, we can assume that $\mathcal W'\cap\mathcal W_\nu=\varnothing$, so that $\mathcal W:=\mathcal W'\cup\mathcal W_\nu$ is a separating neighborhood for $\mathcal E$. Since $\mathcal E$ is an essential family for $\mathcal P_{n}(k)$ there exist $u\in\mathcal P_{n}(k)$ and $\lambda>0$ as in Definition \ref{d:ess}. We claim that there exists a finite, and possibly empty, collection $K_1,...,K_{j_*}$ of closed intervals, which are pairwise disjoint, contained in $[0,1]$ and such that
\begin{align*}
(a)&\quad \left\{
\begin{aligned}
K_j &\subset u^{-1}(\mathcal W_\nu)\cap \big\{{\mathbb S}_k(u,\,\cdot\,)>c_n(k)-\lambda\big\}\\
\partial K_j &\subset\big\{{\mathbb S}_k(u,\,\cdot\,)<c_n(k)\big\}
\end{aligned}\quad\right|\quad \forall\,j\in\{1,...,j_*\},\\[0.5em]
(b)&\quad \big\{{\mathbb S}_k(u,\,\cdot\,)>c_n(k)-\lambda\big\}\setminus K\subset\mathcal W',\quad K:=K_1\cup...\cup K_r.
\end{align*}
The proof of the claim goes as follows. We observe that
\[
[0,1]=\Big(\big\{{\mathbb S}_k(u,\,\cdot\,)>c_n(k)-\lambda\big\}\cap u^{-1}(\mathcal W_\gamma)\Big)\cup\Big(\big\{{\mathbb S}_k(u,\,\cdot\,)<c_n(k)\big\}\cup u^{-1}(\mathcal W')\Big).
\]
We denote by $\{B_a\ |\ a\in A\}$ the set of connected components of the open set $\big\{{\mathbb S}_k(u,\,\cdot\,)>c_n(k)-\lambda\big\}\cap u^{-1}(\mathcal W_\gamma)\subset (0,1)$, where $A$ is some set of indices. The $B_a$'s are pairwise disjoint and each of them is an open interval. Since $[0,1]$ is compact, there exists a finite family $B_{a(1)},...,B_{a(r)}$ such that  
\[
[0,1] = \Big(\bigcup_{j=1}^{r}B_{a(j)}\Big)\cup\Big(\big\{{\mathbb S}_k(u,\,\cdot\,)<c_n(k)\big\}\cup u^{-1}(\mathcal W')\Big).
\]
The claim follows by taking $K_j$ to be any closed interval contained in $B_{a(j)}$ with the property that $\partial K_j\subset \{{\mathbb S}_k(u,\,\cdot\,)<c_n(k)\}\cup u^{-1}(\mathcal W')$.

Now that the claim has been proved, we can apply Proposition \ref{iterationofmountainpasses} with $\{\beta_0,\beta_1\}=\partial K_j$, for all $j=1,...,r$, and obtain paths $v_j:K_j\to\mathcal V^\gamma$ such that
\begin{align}\label{eq:u'j}
v_j|_{\partial K_j}=u|_{\partial K_j},\qquad v_j(K_j)\ \subset\ \big\{{\mathbb S}^\gamma_k\leq {\mathbb S}_k^\gamma(\gamma^\nu)-\lambda'\big\},
\end{align}
for some $\lambda'\in(0,\lambda]$ independent of $j$. We now modify the path $u$ to a path  $v:[0,1]\to \mathcal N^n$ requiring
\begin{align}\label{eq:u'}
(i)\quad v=u\ \  \mathrm{on}\ \ [0,1]\setminus K,\qquad (ii)\quad v|_{K_j}=v_j,\quad \forall\,j=1,...,r.
\end{align}
As $\{0,1\}\cap K=\varnothing$, we see that $v(0)=u(0)=u(1)=v(1)$. Moreover, since $u$ and $v$ coincide on the complement of the open set $\mathcal V^\gamma$, we have
\begin{equation*}
\int_0^1v^*\tau = \int_0^1u^*\tau \neq 0
\end{equation*}
by Lemma \ref{lem:top}, so that $v\in\mathcal P_n(k)$. For the same reason ${\mathbb S}_k(v,\,\cdot\,)={\mathbb S}_k(u,\,\cdot\,)$ on $[0,1]\setminus K$. Let now $j$ be in $\{1,...,r\}$. For all $s\in K_j$, there holds
\[
\lambda'\leq {\mathbb S}_k^{\gamma}(\gamma^\nu)-{\mathbb S}_k^{\gamma}(v(s))=c_n(k)-{\mathbb S}_k(v,s)
\]
where we used \eqref{eq:u'j} and property \textbf{(B2)}. This shows that
\[
K_j\subset \big\{{\mathbb S}_k(v,\,\cdot\,)\leq c_n(k)-\lambda'\big\}.
\]
Applying this argument to every $j$, we conclude that
\[
\big\{{\mathbb S}_k(v,\,\cdot\,)>c_n(k)-\lambda'\big\} \subset v^{-1}(\mathcal W').
\]
Thus, $\mathcal E'$ is also an essential family for $\mathcal P_n(k)$ and the lemma is proved.
\end{proof}

\begin{proof}[Proof of Proposition \ref{lem:main}]
We only need to consider the case in which $\alpha_{k^{*}}$ belongs to a connected component $\mathcal N$ of $\mathcal M$ different from $\mathcal M_0$. We set 
\begin{equation}
I' := \bigcap_{n\in\N}I'_n,\qquad I'_n:=\Big \{k\in I\ \Big | \ c_n\ \text{is differentiable at }k\Big\}.
\label{definizioneJ}
\end{equation}
By the Lebesgue differentiation theorem every function $c_n$ is differentiable almost everywhere, since $c_n$ is monotone by Lemma \ref{lem:mono}. In particular, every $I'_n\subset I$ is a full-measure set. The same is true for $I'$ being a countable intersection of full-measure sets in a space of finite measure.

We claim that for all $k\in I'$ there exist infinitely many closed magnetic geodesics with energy $k$ in $\bigcup_{n}\mathcal N^n$. Suppose by contradiction that there exists $k\in I'$ such that the zero-set of $\eta_{k}$ in $\bigcup_n\mathcal N^n$ consists of finitely many circles
\[\T \cdot \beta_1, \ldots ,\T\cdot \beta_r
\] together with their iterates. In particular, all the vanishing circles in $\bigcup_n\mathcal N^n$ are isolated. We apply Proposition \ref{iterationofmountainpasses} to the orbits $\beta_1,...,\beta_r$ and get numbers $\nu(\beta_1),...,\nu(\beta_r)\in\N$. We define the family
\[
\mathcal F:=\bigcup_{1\leq j\leq r}\bigcup_{1\leq \nu\leq\nu(\beta_j)}\T\cdot\beta_j^\nu.
\]
Since $\mathcal N^{n_1}\neq \mathcal N^{n_2}$ if $n_1\neq n_2$ and $\mathcal F$ is a finite union of circles, there exists $n\in\N$ such that $\mathcal N^n\cap\mathcal F=\varnothing$. By definition of $I'$ and the fact that all vanishing circles in $\mathcal N^n$ are isolated, we can apply Lemma \ref{Struwesurface} and get an essential family $\mathcal E_n(k)\subset\mathcal N^n$ for $\mathcal P_n(k)$, which is a finite union of circles. By Lemma \ref{lem:el}
\[
\mathcal E_n(k)\cap\mathcal F=\varnothing
\]
is an essential family, as well. Since an essential family cannot be empty by Remark \ref{r:empty}, we get a contradiction, which proves the proposition.
\end{proof}

\bibliographystyle{amsalpha}      
\bibliography{wes}   

\end{document}